\documentclass{amsart}

\usepackage{amssymb,amsthm,amstext,amsfonts,amsmath}
\usepackage[inline]{enumitem}

\newcommand{\CDH}{\ensuremath{\mathsf{CDH}}}
\newcommand{\filt}{\mathcal{F}}
\newcommand{\powerw}{\mathcal{P}(\omega)}
\newcommand{\cantor}{2\sp\omega}
\newcommand{\pair}[1]{\langle #1 \rangle}
\newcommand{\R}{\mathbb{R}}
\newcommand{\Q}{\mathbb{Q}}
\newcommand{\cube}{\mathcal{Q}} 
\newcommand{\cubein}{\mathsf{s}} 
\newcommand{\id}[1]{\mathbf{1}_{#1}}
\newcommand{\A}{\mathcal{A}}

\newcommand{\p}{\mathfrak{p}}
\newcommand{\homeo}[1]{\mathcal{H}(#1)}
\newcommand{\cont}{\mathfrak{c}}
\newcommand{\G}{\mathcal{G}}
\newcommand{\poset}{\mathbb{P}}

\newtheoremstyle{theorem}
     {11pt}
     {11pt}
     {}
     {}
     {\bfseries}
     {}
     {.5em}
     {\noindent\thmnumber{#2}. \thmname{#1}{\rm\thmnote{#3}}}

\theoremstyle{theorem}

\newtheorem{thm}{Theorem}[section]
\newtheorem{lemma}[thm]{Lemma}
\newtheorem{propo}[thm]{Proposition}
\newtheorem{coro}[thm]{Corollary}

\newtheorem{ques}[thm]{Question}

\title{Countable dense homogeneity of function spaces}

\author[Hern\'andez-Guti\'errez]{Rodrigo Hern\'andez-Guti\'errez}
\address[Hern\'andez-Guti\'errez]{Departamento de Matem\'aticas, Universidad Aut\'onoma Metropolitana campus Iztapalapa, Av. San Rafael Atlixco 186, Col. Vicentina, Iztapalapa, 09340, Mexico city, Mexico}
\email[Hern\'andez-Guti\'errez]{rodrigo.hdz@gmail.com}

\date{\today}


\begin{document}

\begin{abstract}
  In this paper we consider the question of when the space $C_p(X)$ of continuous real-valued functions on $X$ with the pointwise convergence topology is countable dense homogeneous. In particular, we focus on the case when $X$ is countable with a unique non-isolated point $\infty$. In this case, $C_p(X)$ is countable dense homogeneous if and only if the filter of open neighborhoods of $\infty$ is a non-meager $P$-filter.
 \end{abstract}

 \maketitle
 
\section{Introduction}
 
 All spaces considered are assumed to be Tychonoff.
 
A space $X$ is countable dense homogeneous (\CDH{}, henceforth) if $X$ is separable and whenever $D,E\subset X$ are countable dense subsets, there is a homeomorphism $h\colon X\to X$ such that $h[D]=E$. Among examples of \CDH{} spaces we have the Euclidean spaces, the Hilbert cube and the Cantor set. For updated surveys on \CDH{} spaces, see sections 14, 15 and 16 of \cite{arh-vm-homogeneity}; and \cite{hru-vm-cdh_survey}.

One of the most notable open problems in the theory of \CDH{} spaces is the existence of metric \CDH{} spaces that are not Polish; this is Problem 6 in \cite{openprobCDH}. This problem was solved in \cite{FarahHrusakMartinez05} where the authors construct a meager-in-itself \CDH{} subspace of the reals. Later, another example of a  \CDH{} non-Polish subspace of the reals that is a Baire space was given in \cite{hg-hr-vm}.

Another related result is that $\R\sp\kappa$ is \CDH{} if and only if $\kappa<\p$. This result was proved in two steps: first Stepr\=ans and Zhou (\cite{step-zhou}) proved sufficiency and later Hru\v s\'ak and Zamora-Avil\'es (\cite{hrus-zam}) proved necessity.

Hence, it is natural to consider the space $C_p(X)$ of real-valued continuous functions defined on $X$, with the topology of pointwise convergence. In 2016, after my first seminar talk at UAM Iztapalapa, I was asked the following question.

\begin{ques}(Vladimir Tkachuk)\label{main-ques}
  When is $C_p(X)$ \CDH{}?
 \end{ques}
 
We start with some general remarks in Section \ref{gen-obs}. After this, in Section \ref{section-filterspace}, we define a class of spaces to which we will restrict our attention for the rest of the paper. The rest of the paper is devoted to proving the next result.

\begin{thm}\label{thm-main}
 Let $X$ be a countable space with a unique non-isolated point $\infty$. Then $C_p(X)$ is \CDH{} if and only if the neighborhood filter of $\infty$ is a non-meager $P$-filter.
\end{thm}

 When the author of this paper was working on the proof of sufficiency in Theorem \ref{thm-main}, he consulted the paper \cite{step-zhou} to look for ideas for the proof. It was then when this author noticed that the proof of Theorem 4 of that paper had a gap. Theorem 4 of \cite{step-zhou} is precisely the statement that $X\sp\kappa$ is \CDH{} when $X\in\{[0,1], 2\sp\omega,\R\}$ and $\kappa<\mathfrak{p}$. In that paper the authors use a $\sigma$-centered poset that forces a homeomorphism $h:X\sp\kappa\to X\sp\kappa$ and their intention is to prove that $h[D]=E$ when $D$ and $E$ are two given countable dense sets. However, after analyzing the proof carefully, it does not seem that the equality $h[D]=E$ necessarily hold.

After completing the proof of Theorem \ref{thm-main}, the author of this paper noticed that he could provide a correct proof of the Stepr\={a}ns-Zhou result by using a similar argument. Thus, in section \ref{correction} we include a such a proof

\section{General remarks}\label{gen-obs}

First, let us state some properties of $C_p(X)$. Recall that $C_p(X)$ is a subspace of $\R\sp{X}$ with the product topology.

\begin{thm}\label{Cp-facts}
  Let $X$ be a Tychonoff space. \begin{enumerate}[label=(\alph*)]
   \item \cite[I.1.5, p. 26]{arkh-top_func_spaces} $C_p(X)$ is separable if and only if there is a separable metrizable space $Y$ and a continuous bijection $f:X\to Y$.
   \item \cite[I.3.4, p. 32]{arkh-top_func_spaces} If $C_p(X)$ is a Baire space, then every compact subset of $X$ is finite.
   \item \cite[I.1.1, p. 26]{arkh-top_func_spaces} $C_p(X)$ is metrizable if and only if $X$ is countable
  \end{enumerate} 
 \end{thm}
 
 For separable metric spaces, the Baire category theorem gives us some constraints on which spaces can be \CDH{}. First, every \CDH{} space is a topological sum of homogeneous \CDH{} spaces. A homogeneous space can be either \emph{meager} (countable union of nowhere dense sets) or a Baire space. If $X$ is meager, separable and metrizable, it was proved by Fitzpatrick and Zhou \cite[Lemma 3.2]{fitz_zhou-CDH_baire} that $X$ has a countable dense set that is of type $G_\delta$. From this, it follows that a meager, separable and metrizable \CDH{} space has to be a $\lambda$-set, that is, all countable subsets are of type $G_\delta$.
 
 If $X$ is countable, then $C_p(X)$ is metrizable and these Baire category results apply. Since $C_p(X)$ contains topological copies of the reals, it cannot be a $\lambda$-set. Thus, if $X$ is a countable metric space with at least one non-isolated point, then $C_p(X)$ is meager (because $X$ contains a non-trivial convergent sequence) and cannot be \CDH{}. Thus, the following holds.
 
 \begin{propo}\label{C_p-countable_metric}
  Let $X$ be a countable metric space. Then, $C_p(X)$ is \CDH{} if and only if $X$ is discrete. 
  \end{propo}
  
  We can also say something of $\sigma$-compact spaces using the following result of Tkachuk.

 \begin{thm}\cite[3.16]{tk-no_conv_seq}\label{Cp-dense_noseq}
 Let $X$ be an uncountable, $\sigma$-compact, separable metric space. Then there exists a countable dense subset of $C_p(X)$ that has no non-trivial convergent sequences.
 \end{thm}
 
 \begin{coro}
  Let $X$ be an uncountable, $\sigma$-compact, separable metric space. Then $C_p(X)$ is not \CDH{}.
 \end{coro}
  \begin{proof}
  By Theorem \ref{Cp-dense_noseq}, $C_p(X)$ has a countable dense set $D$ with no non-trivial convergent sequences. Since $C_p(X)$ contains copies of the reals, it also contains many non-trivial convergent sequences; choose one $S$. Then $D\cup S$ is also countable dense. But $D$ and $D\cup S$ are not homeomorphic. The conclusion follows.
 \end{proof}
 
 For the sake of completeness, we include the following questions which the author was not able to answer.
  
 \begin{ques}\label{question-Polish}
  Let $X$ be Polish and uncountable (e.g., the irrationals). Is there a countable dense subset of  $C_p(X)$ that has no non-trivial convergent sequences?
 \end{ques}
 
 \begin{ques}\label{question-metrizable}
  Is there any non-discrete metrizable space $X$ such that $C_p(X)$ is \CDH{}?
 \end{ques}
 
 Let us focus our attention on the case when $X$ is countable but not metrizable. In this case, $C_p(X)$ can be embedded as a dense subset of the infinite-dimensional metric space $\R\sp\omega$. So perhaps, it is possible to find an example of an $X$ where it is possible to prove that $C_p(X)$ is \CDH{} \emph{with respect to} $\R\sp\omega$, that is, given $D,E\subset C_p(X)$ countable dense it might be possible to construct a homeomorphism $h:\R\sp\omega\to\R\sp\omega$ such that $h[C_p (X)]=C_p(X)$ and $h[D]=E$. 
 
 After the discussion regarding Baire category above, the first natural question is whether there is a countable space $X$ such that $C_p(X)$ is a Baire space. It turns out that this question has already been answered in the affirmative. In the next section we discuss this in more detail.

\section{Spaces with a unique non-isolated point}\label{section-filterspace}

Given a filter $\filt\subset\powerw$, consider the space $\xi(\filt)=\omega\cup\{\infty\}$, where every point of $\omega$ is isolated and every neighborhood of $\infty$ is of the form $\{\infty\}\cup A$ with $A\in\filt$. All of our filters will be \emph{free}, that is, they contain the Fr\'echet filter. In this case, $\xi(\filt)$ is Hausdorff and $\infty$ is not isolated. When $\filt$ is the Fr\'echet filter, $\xi(\filt)$ is homeomorphic to a convergent sequence. Clearly, a space $X$ is homeomorphic to a space of the form $\xi(\filt)$ if and only if $X$ is a countable space with a unique non-isolated point.

A filter in $\omega$ is a subset of $\powerw$ so it can also be viewed as a subset of the Cantor set via the identification of a subset with its characteristic function. It turns out that a characterization of when $C_p(\xi(\filt))$ is a Baire space is related to this topology. 

 \begin{lemma}\cite[Theorem 6.5.6, p. 391]{vm-inf_dim_funct_spaces}\label{Cp_Baire_char}
  Let $\filt\subset\powerw$ be a free filter. Then $C_p(\xi(\filt))$ is a Baire space if and only if $\filt$ is non-meager (as a subset of the Cantor set).
 \end{lemma}

Naturally, it is possible to ask when a free filter is \CDH{}. At first, this seems to be another problem entirely but surprisingly it will be related to our problem. Recall that a filter $\filt\subset\powerw$ is a $P$-filter if whenever $\{F_n:n<\omega\}\subset\filt$ there exists $F\in\filt$ such that $F\setminus F_n$ is finite for all $n<\omega$.
 
 \begin{thm}\cite{hg-hr-CDH,kunen-medini-zdomskyy}\label{char-nonmeagerP}
  Let $\filt\subset\powerw$ be a free filter, and consider the Cantor set topology in $\powerw$. Then $\filt$ is \CDH{} if and only if $\filt$ is a non-meager $P$-filter.
 \end{thm}
 
 \emph{Non-meager $P$-filters} are a natural generalization of $P$-points; it is still an open question whether they exist in ZFC but if they do not exist, then there is an inner model with a large cardinal, see \cite[section 4.4.C]{bart}.
 
 In the remainder of this section the author would like to convince the reader that the spaces $C_p(\xi(\filt))$ and $\filt\sp{\omega}$ are very similar. Perhaps one of the first objections to this claim is the fact that $C_p(\xi(\filt))$ is connected (in fact, arcwise connected) dense subset of the Hilbert cube and $\filt\sp{\omega}$ is a zero-dimensional space. So let us consider the following set that is a zero-dimensional space
 $$
 C_p(X,\cantor)=\left\{f\in(\cantor)\sp{X}:f\textrm{ is continuous}\right\},
 $$
 where $X$ is a countable regular topological space. We will show that $C_p(\xi(\filt),\cantor)$ and $\filt\sp{\omega}$ are not only similar but also homeomorphic (Theorem \ref{Cp-zerodim}).

 \begin{lemma}\label{lemmasum}
  Let $\filt\subset\powerw$ be a free filter. Then $\filt\times2$ is homeomorphic to $\filt$. 
 \end{lemma}
 \begin{proof}
  If $\filt$ is the Fr\'echet filter, then it is homeomorphic to the rationals $\mathbb{Q}$ and $\mathbb{Q}\times 2\approx\mathbb{Q}$. So assume there is $X\in\filt$ such that $\omega\setminus X$ is infinite. Let $\omega\setminus X=\{x_m:m<\omega\}$ be an enumeration. 
  
  Consider $F\subset\cantor$ the set of characteristic functions of elements of $\filt$, which is homeomorphic to $\filt$. Consider the clopen set $U_i=\{f\in\cantor:f(x_0)=i\}$ for $i\in 2$. Then $\{U_0\cap F,U_1\cap F\}$ is a partition of $F$; we will prove that $U_0\cap F\approx U_1\cap F\approx F$.
  
  First, consider $\varphi:\cantor\to\cantor$ such that $\varphi(f)(n)=f(n)$ if $n\neq x_0$ and $\varphi(f)(n)=1-f(n)$ if $n= x_0$, for all $f\in\cantor$. Then $\varphi$ is a homeomorphism and $\varphi[U_0]=U_1$. Notice that also $\varphi[F]=F$ because $\filt$ is closed under finite changes. Then $\varphi$ is a witness of $U_0\cap F\approx U_1\cap F$.
  
  Now we define $\psi:\cantor\to\cantor$ for $f\in\cantor$ as follows:
  $$
  \psi(f)(n)=\left\{
  \begin{array}{ll}
   f(n), & \textrm{ if } n\in X,\\
   0, & \textrm{ if } n=x_0,\textrm{ and}\\
   f(x_{k-1}), & \textrm{ if } n=x_k, k>0.\\
  \end{array}\right.
  $$
  It is easy to see that $\psi$ is an embedding with image $U_0$. From $X\in\filt$ it is easy to see that $\psi[F]=U_0\cap F$. Thus, $\psi$ witnesses that $F\approx U_0\cap F$.
  
  Then we have expressed $F$ as the union of two proper clopen subsets $U_0\cap F$ and $U_1\cap F$, each homeomorphic to $F$. This easily implies the statement of the lemma.
 \end{proof}
 
 For $\filt\subset\powerw$, let $\filt\sp{(\omega)}$ be the set of all subsets of $\omega\times\omega$ of the form $\bigcup\{\{n\}\times A_n:n<\omega\}$, where $\{A_n:n<\omega\}\subset\filt$. It is easy to see that $\filt\sp{(\omega)}$ is a filter on $\omega\times\omega$ and it is homeomorphic to $\filt\sp{\omega}$ with the product topology.
 
 \begin{lemma}\label{powerfilter}
  Let $\filt\subset\powerw$ be a free filter. 
  \begin{enumerate}[label=(\alph*)]
   \item $\filt$ is non-meager if and only if $\filt\sp{(\omega)}$ is non-meager.
   \item $\filt$ is a $P$-filter if and only if $\filt\sp{(\omega)}$ is a $P$-filter.
  \end{enumerate}
 \end{lemma}
 \begin{proof}
  The forward implication of both (a) and (b) was proved by Shelah \cite[Fact 4.3, p. 327]{shelah-proper} (see also \cite[Lemma 2.5]{hg-hr-CDH} for a proof). So we only have to prove the converse implications.
  
  First, notice that if $X\subset\cantor$ is a meager space, then $X\sp\omega$ is also meager. Write $X$ as a union of countably many closed and nowhere dense sets $X=\bigcup\{F_n:n<\omega\}$. Then $X\sp\omega=\bigcup\{F_n\times X\sp{\omega\setminus\{0\}}:n<\omega\}$ where each $F_n\times X\sp{\omega\setminus\{0\}}$ is closed and nowhere dense in $X\sp\omega$. From this, it follows that if $\filt\sp{(\omega)}$ is non-meager, then $\filt$ is non-meager as well.
  
  Now, assume that $\filt\sp{(\omega)}$ is a $P$-filter. Let $\{A_n:n<\omega\}\subset\filt$; we wish to obtain a pseudointersection. For each $n<\omega$, let $B_n=\omega\times A_n$; this is an element of $\filt\sp{(\omega)}$. Let $B$ be a pseudointersection of $\{B_n:n<\omega\}$. Now, define $$A=\{k\in\omega:\exists m\in\omega\ (\pair{m,k}\in B)\}.$$ Notice that $B\subset \omega\times A$. Thus, $\omega\times A\in\filt\sp{(\omega)}$ and this implies that $A\in\filt$. We will prove that $A$ is a pseudointersection of $\{A_n:n<\omega\}$. 
  
  Fix $n<\omega$. Now, consider $k\in A\setminus A_n$. Then there is a minimal $m_k<\omega$ with $\pair{m_k,k}\in B$ and also $\pair{m,k}\notin B_n$ for all $m<\omega$. Thus, $\pair{m_k,k}\in B\setminus B_n$. Since there are only finitely many elements in $B\setminus B_n$, there are only many finitely elements in $A\setminus A_n$. This proves that $A$ is indeed the pseudointersection we were looking for.
 \end{proof}
 
 \begin{thm}\label{Cp-zerodim}
  Let $\filt\subset\powerw$ be a free filter. Then $C_p(\xi(\filt),\cantor)$ is homeomorphic to $\filt\sp{\omega}$.
 \end{thm}
  \begin{proof}
  First, notice that $C_p(\xi(\filt),\cantor)$ is homeomorphic to $(C_p(\xi(\filt),2))\sp\omega$. Moreover, the function that sends each $f\in C_p(\xi(\filt),2)$ to the function $x \mapsto 1-f(x)$ is a homeomorphism. Thus, we may only consider functions that send the limit point $\filt$ to $1$:
  $$
  (C_p(\xi(\filt),2))\sp\omega\approx(\{f\in C_p(X,2):f(\filt)=1\}\times 2)\sp\omega\textrm{.}
  $$
  Each function in $C_p(\xi(\filt),2)$ is completely defined by its restriction to $\omega$. The restriction of the set of functions $\{f\in C_p(\xi(\filt),2):f(\filt)=1\}$ is then the set of characteristic functions of elements of $\filt$, which is homeomorphic to $\filt$ as a subset of the Cantor set. Then $C_p(\xi(\filt),2\sp\omega)\approx (\filt\times 2)\sp\omega$ and Lemma \ref{lemmasum} implies that $C_p(\xi(\filt),2\sp\omega)\approx \filt\sp\omega$.    
 \end{proof}
 
 Applying Lemma \ref{powerfilter} and Theorem \ref{char-nonmeagerP}, we obtain the following result. 
 
 \begin{coro}
   Let $\filt\subset\powerw$ be a free filter. Then $C_p(\xi(\filt),\cantor)$ is \CDH{} if and only if $\filt$ is a non-meager $P$-filter.
 \end{coro}
 
 Now we turn back to the original problem. The rest of our paper will consist on giving a proof of the following result, which is a more explicit restatement of Theorem \ref{thm-main}.
 
 \begin{thm}\label{main-thm}
  Let $\filt\subset\powerw$ be a free filter. Then $C_p(\xi(\mathcal{F}))$ is \CDH{} if and only if $\filt$ is a  a non-meager $P$-filter.
 \end{thm}
 
 \begin{coro}\label{solution-tkachuk_question}
  There are models of \textrm{ZFC} in which there exists a countable infinite space $X$ with a unique non-isolated point such that $C_p(X)$ is  \CDH{}.
 \end{coro}
 \begin{proof}
  According to \cite[section 4.4.C]{bart} there are models of \textrm{ZFC} in which there exists a non-meager $P$-filter $\filt$. By Theorem \ref{main-thm}, the space $X=\xi(\filt)$ is as promised.  
 \end{proof}
 
 Once we prove Theorem \ref{main-thm}, the space $X$ from Corollary \ref{solution-tkachuk_question} will be the first consistent example of a metrizable, non-Polish, Baire, arcwise connected, infinite-dimensional, \CDH{} space.
 
 Before we close this section, we include other definitions that we will need for the proof of Theorem \ref{main-thm}.
 
 The \emph{Hilbert cube} is the countable infinite product of closed intervals of the reals; we will find it convenient to work with $\cube=[-1,1]\sp\omega$. The \emph{pseudointerior} of $\cube$ is $\cubein=(-1,1)\sp\omega$ and the \emph{pseudoboundary} is $B(\cube)=\cube\setminus \cubein$. Consider the set of functions in the Hilbert cube that $\filt$-converge to $0$:
 $$
 K_\filt=\{f\in\cube:\forall m\in\omega\ \{n\in\omega:\lvert f(n)\rvert<2\sp{-m}\}\in\filt\},
 $$
 and those that only take values in the pseudointerior
 $$
 C_\filt=K_\filt\cap\cubein.
 $$
 
 From \cite[Lemma 2.1]{marciszewski-analytic_coanalytic}, it easily follows that $C_\filt$ is homeomorphic to $C_p(\xi(\filt))$. Thus, instead of working with the pair of spaces $C_p(\xi(\filt))\subset\R\sp{\xi(\filt)}$, we will work with $C_\filt\subset\cubein$.
 
 A family of closed, non-empty subsets $\A=\{A_s:s\in 2\sp{<\omega}\}$ of a compact space $X$ will be called a \emph{Cantor scheme} on $X$ if 
 \begin{itemize}
  \item[(a)] for every $s\in2\sp{<\omega}$, $A_s=A_{s\sp\frown0}\cup A_{s\sp\frown1}$, and
  \item[(b)] for every $f\in2\sp\omega$, $\bigcap\{A_{f\restriction_n}:n<\omega\}$ is exactly one point.
 \end{itemize}
 Let $\mathbf{1}\in 2\sp\omega$ be the constant function with value $1$ and consider the set 
 $$\sigma=\{x\in2\sp{<\omega}:\forall i\in{dom}(s)\ [s(i)=1]\}$$
 of all finite sequences that are either constant equal to $1$ or empty.

\section{Proof of necessity in Theorem \ref{main-thm}}

In this section, we prove that if $C_p(\xi(\filt))$ is \CDH{}, then the filter $\filt$ must be a non-meager $P$-filter.

In \cite[Theorem 10]{kunen-medini-zdomskyy}, the authors give several topological characterizations of non-meager $P$-filters. One of then is what they call the Miller property: A space $X$ has the \emph{Miller property} if whenever $Q\subset X$ is crowded and countable, there exists a crowded subset $C\subset Q$ such that $C$ has compact closure in $X$.

 \begin{lemma}\cite[Lemma 7]{kunen-medini-zdomskyy}\label{miller-filters}
  Let $\filt\subset\powerw$ be a free filter that is non-meager. Let $P$ be a countable crowded subspace of $\filt$ such that $P$ has a pseudointersection in $\filt$. Then there exists a crowded $Q\subset P$ and $z\in\filt$ such that for every $x\in Q$, $z\subset x$.
 \end{lemma}
 
 The following result is a function space version of the arguments given in \cite{kunen-medini-zdomskyy}.
 
  \begin{propo}\label{semiMiller}
  Let $\filt\subset\powerw$ be a free filter that is non-meager. There is a countable dense set $D\subset C_p(\xi(\filt))$ such that for every crowded $P\subset D$ there is a crowded subset $Q\subset P$ that has compact closure in $C_p(\xi(\filt))$.
 \end{propo}
 \begin{proof}
  We will work with $C_\filt$. We define $D$ to be the set of functions $f:\omega\to (-1,1)$ such that there is a finite set $s_f\in[\omega]\sp{<\omega}$ such that $f(x)\in \Q\setminus\{0\}$ if $x\in s_f$ and $f(x)=0$ if $x\notin s_f$. Clearly, $D$ is countable and dense in $\cubein$. Moreover, $D\subset C_\filt$ because every function in $D$ is equal to $0$ in a cofinite set.
  
  Let $P\subset D$ be crowded. First, we will find a sequence $\{J_n:n<\omega\}$ of closed subintervals of $[-1,1]$ and a crowded subset $P\sp\prime\subset P$ such that $P\sp\prime\subset\prod_{n<\omega}J_n\sp\circ$. We define these intervals and $P\sp\prime=\{p_n:n<\omega\}$ recursively. It is easy to construct a function $\psi:\omega\to\omega$ such that $\psi(n)<n$ if $n>0$ and $\psi\sp{-1}(n)$ is infinite for all $n<\omega$. The reason we are considering this function is to make $\{p_m:m<\omega,\psi(m)=n\}$ a sequence converging to $p_n$ for all $n<\omega$. Let $\rho$ be a metric in $\cubein$.
  
  For $n=0$, let $p_0\in P$ be arbitrary and let $J_0$ be any non-degenerate closed interval centered at $0$ such that $p_0(0)\in J_0\sp\circ$ and $J_0\subset(-1,1)$. So assume that $0<n<\omega$ and we have constructed $\{p_i:i<n\}\subset P$ such that
  \begin{itemize}
   \item if $i,j<n$, then $p_i(j)\in J_j\sp\circ$ and $J_j\subset(-1,1)$; and
   \item if $i,j<n$ and $i=\psi(j)$, then $\rho(p_i,p_j)<1/(j+1)$.
  \end{itemize}
  So we need to choose $p_n$ and $J_n$. The open set $U_n=J_0\sp\circ\times\cdots\times J_{n-1}\sp\circ\times(-1,1)\sp{\omega\setminus n}$ contains $\{p_i:i<n\}$. Then, choose $p_n\in J_0\sp\circ\times\cdots\times J_{n-1}\sp\circ\times(-1,1)\sp{\omega\setminus n}$ such that $\rho(p_{\psi(n)},p_n)<1/(n+1)$. Finally, let $J_{n+1}\subset(-1,1)$ be a closed interval centered at $0$ that contains $\{p_i(n):i\leq n\}$ in its interior. This completes the construction of $P\sp\prime$.
  
 Fix $n<\omega$ for the moment. It is not hard to construct a centered scheme $\A\sp{n}=\{A_s\sp{n}:s\in2\sp{<\omega}\}$ with the following properties:
  \begin{enumerate}[label=(\roman*)]
   \item $A_\emptyset\sp{n}=J_n$,
   \item if $s\in\sigma$, then there are $x$ and $y$ with $0<x<y$, $A_s\sp{n}=[-y,y]$, $A_{s\sp\frown1}\sp{n}=[-x,x]$ and $A_{s\sp\frown0}\sp{n}=[-y,-x]\cup[x,y]$,
   \item if $s\in\sigma$, then there are $x$ and $y$ with $0<x<y$, $A_{s\sp\frown\pair{0,0}}\sp{n}=[-y,-x]$ and $A_{s\sp\frown\pair{0,1}}\sp{n}=[x,y]$\
   \item if $s\in2\sp{<\omega}\setminus\sigma$ and there are $a$ and $b$ with $a<b$ and $A_s\sp{n}=[a,b]$, then there exists $x\in(a,b)$ such that $A_{s\sp\frown0}\sp{n}=[a,x]$ and $A_{s\sp\frown1}\sp{n}=[x,b]$,
   \item if $s\in\sigma\setminus\{\emptyset\}$, then $A_s\sp{n}$ has length less than $1/\mathrm{dom}(s)$,
   \item if $s\in2\sp{<\omega}$, $\{p_i(n):i<\omega\}$ is disjoint from the set of endpoints of $A_s\sp{n}$.  
  \end{enumerate}

  Then, define a continuous function $\varphi_n:2\sp\omega\to J_n$ such that $\varphi_n(x)$ is the unique point in the set $\bigcap_{i<\omega}{A_{x\restriction i}\sp{n}}$. Notice that the following properties hold.
  \begin{enumerate}[label=(\arabic*)]
   \item $\varphi_n(\mathbf{1})=0$, and
   \item for all $0<i<\omega$, $\varphi_n[\mathbf{1}\!\!\restriction_i]$ is a closed interval of length $<1/i$ that contains $0$ in its interior.
   \end{enumerate}
  
  Consider the continuous function $\varphi:(2\sp\omega)\sp\omega\to\prod_{n<\omega}{J_n}$ defined as 
  $$
  \varphi(x)=\{\pair{n,\varphi_n(x(n))}:n<\omega\}.
  $$
  It is easy to see that $|\varphi\sp{\leftarrow}{(p_i)}|=1$ for all $i<\omega$. 
  
  Let $A=\{x\in (2\sp\omega)\sp\omega:\varphi(x)\in P\sp\prime\}$. Then $\varphi\!\!\restriction_A:A\to P\sp\prime$ is a homeomorphism. Recall that $(2\sp\omega)\sp\omega$ is cannonically homeomorphic to $\mathcal{P}(\omega\times\omega)$ under the function
  $$
  f\mapsto\{\pair{n,m}\in\omega\times\omega:f(m)(n)=1\}.
  $$
  Let $B=\{x_n:n<\omega\}\subset\omega\times\omega$ be the element corresponding to $A$ under this homeomorphism. Since every element of $P\sp\prime$ is equal to $0$ in cofinitely many coordinates, it is easy to argue that for every $n<\omega$, $x_n$ has the property that there is $m_n<\omega$ such that $\omega\times(\omega\setminus m_n)\subset x_n$.
  
  Since $\filt$ contains all cofinite sets, $B\subset\filt\sp{(\omega)}$. By (a) in Lemma \ref{powerfilter}, $\filt\sp{(\omega)}$ is non-meager. Also, $B$ is crowded since it is homeomorphic to $P\sp\prime$. Notice that $z=\{\pair{n,k}:n<\omega,\ \forall i\leq n\ (m_i\leq k)\}$ is a pseudointersection of $B$ in $\filt\sp{(\omega)}$. By Lemma \ref{miller-filters}, there is a crowded subset $C\subset B$ and $y\in\filt\sp{(\omega)}$ such that $y\subset x$ for all $x\in C$.
  
  Now, we have to go back to $(\prod_{n<\omega}{J_n})\cap C_\filt$. Let $Q$ be the set that corresponds to $C$ under the homeomorphisms considered. Then $Q$ is a crowded subset of $P$ and we claim that it has compact closure in $C_\filt$. Since $\prod_{n<\omega}{J_n}$ is compact, the closure of $Q$ in $\cubein$ is compact. Then we only have to prove that the $\cubein$-closure of $Q$ is contained in $C_\filt$.
  
  For each $1\leq n<\omega$, let $y_n=\{k<\omega:\forall i\leq n (\pair{i,k}\in y)\}$. Since $y\in\filt\sp{(\omega)}$, it follows that $y_n\in\filt$ for all $n\leq \omega$. If $q\in Q$ and $k\in y_n$ for some $n<\omega$, then $q(k)\in \varphi_n[\mathbf{1}\!\!\restriction_{n}]$. So if $f$ is in the pointwise closure closure of $Q$, then $f(k)\in \varphi_k[\mathbf{1}\!\!\restriction_{n}]$ for $k\in y_n$. By condition (2),  $\varphi_k[\mathbf{1}\!\!\restriction_{n}]$ has diameter $<1/n$ and contains $0$. Thus, for all $f$ in the $\cubein$-closure of $Q$,
  $y_n\subset\{k<\omega:\lvert f(k) \rvert< 1/n)\}$, so $\{k<\omega:\lvert f(k) \rvert< 1/n\}\in\filt$. Since this holds for all $n<\omega$, it follows that $f\in C_\filt$. This completes the proof.
 \end{proof}

 \begin{proof}[Proof of necessity in Theorem \ref{main-thm}]
  Let $\filt\subset\powerw$ be a free filter such that $C_p(\xi(\filt))$ is \CDH{}. Since $C_p(\xi(\filt))$ contains a copy of $\R$, we can argue like in Proposition \ref{C_p-countable_metric} that it cannot be meager. Since $C_p(\xi(\filt))$ is homogeneous, it follows that it is a Baire space. By Lemma \ref{Cp_Baire_char}, we obtain that $\filt$ is non-meager.
  
  We will next prove that $C_p(\xi(\filt))$ has the Miller property. Let $Q\subset C_p(\xi(\filt))$ be crowded and countable. Let $D$ be the countable dense set from Proposition \ref{semiMiller}. Since $C_p(\xi(\filt))$ is\CDH, there is a homeomorphism $H:C_p(\xi(\filt))\to C_p(\xi(\filt))$ with $H[D\cup Q]=D$. Then, by the property of $D$, there is a crowded $C\subset H[Q]$ that has compact closure in $C_p(\xi(\filt))$. Then $H\sp{\leftarrow}[C]$ is a countable crowded subset of $Q$ with compact closure in $C_p(\xi(\filt))$.  
  
  Let us define $e:\filt\to C_p(\xi(\filt))$ in the following way. For every $A\in\filt$, $e(A):\xi(\filt)\to\R$ is the function such that $e(A)(\filt)=1$ and $e(A)\!\!\restriction_\omega:\omega\to 2$ is the characteristic function of $A$. Clearly, $e$ is an embedding and it is not hard to see that $e[\filt]$ is closed. Since the Miller property is clearly hereditary to closed subspaces, we conclude that $\filt$ has the Miller property. By the characterization of \cite[Theorem 10]{kunen-medini-zdomskyy}, this implies that $\filt$ is a non-meager $P$-filter.
 \end{proof}

\section{Proof of sufficiency in Theorem \ref{main-thm}}

In this section we prove the other direction of our main Theorem \ref{main-thm}. What we will actually do is to prove that $C_\filt$ is \CDH{} \emph{with respect to} $\cube$ whenever $\filt\subset\powerw$ is a non-meager $P$-filter. This means that given $D,E$ countable dense subsets of $C_\filt$, there is a homeomorphism $h:\cube\to\cube$ such that $h[D]=E$ and moreover, $h[C_\filt]=C_\filt$.

In order to achieve this, we will use the fact that the theory of homeomorphisms in compact metric spaces is well-understood. For a compact space $X$, $\homeo{X}$ denotes the set of all autohomeomorphisms of $X$. Given a metric $\rho$ on $X$ it is possible to define a metric on $\homeo{X}$ by $\rho(f,g)=\sup\{\rho(f(x),g(x))\colon x\in X\}$, when $f,g\in\homeo{X}$, which makes $\homeo{X}$ a complete metric space. This discussion can be found in \cite{vm-inf_dim_funct_spaces}. 

One of the advantages of considering $\homeo{X}$ as a topological space is that we can construct complicated homeomorphisms by taking limits of simpler homeomorphisms. The main tool that we use in this paper is the following result, which is called the \emph{Inductive Convergence Criterion}.

\begin{thm}\cite[1.6.2]{vm-inf_dim_funct_spaces}\label{inductive_conv}
  Let $\pair{X,\rho}$ be a compact metric space and $\{h_n:n<\omega\}\subset\homeo{X}$. Then, for every $n<\omega$ there exists a number $\epsilon_n>0$ that only depends on $\{h_k:k<n\}$ such that the following implication holds: if $\rho(h_n,\id{X})<\epsilon_n$ for all $n<\omega$, then $h=\lim_{n\to\infty}(h_n\circ\cdots\circ h_0)$ exists and belongs to $\homeo{X}$.
 \end{thm}
 
 Thus, when we construct a homeomorphism by recursion, at every step $n<\omega$ we have to consider the value of $\epsilon_n$ in order to choose $h_n$. However, in our constructions we will not worry about the precise value of this $\epsilon _n$ and simply carry out the construction.
 
 Another ingredient we need is some criterion that will ensure that a homeomorphism $h\in\homeo{\cube}$ will be such that $h[C_\filt]=C_\filt$.
 
 \begin{lemma}\label{homeo-restricts1}
  Let $\filt\subset\powerw$ be a free filter. Assume that $h\in\homeo{\cube}$ has the property that for every $x\in\cube$, $(x-h(x))\in C_\filt$. Then $h[K_\filt]=K_\filt$.
 \end{lemma}
 \begin{proof}
  Let $m\in\omega$. Then $A=\{n\in\omega:\lvert (x-h(x))(n)\rvert<2\sp{-(m+1)}\}\in\filt$. If $x\in C_\filt$, then $B=\{n\in\omega: \lvert x(n)\rvert<2\sp{-(m+1)}\}\in\filt$. By the triangle inequality it follows that $A\cap B\subset\{n\in\omega:\lvert h(x)(n)\rvert<2\sp{-m}\}$ so $\{n\in\omega:\lvert h(x)(n)\rvert<2\sp{-m}\}\in\filt$. This shows that $h[C_\filt]\subset C_\filt$, the other inclusion follows in a similar way.
 \end{proof}

 We will use the following special version of Lemma \ref{homeo-restricts1}.

 \begin{lemma}\label{homeo-restricts2}
  Let $\filt\subset\powerw$ be a free filter, let $h\in\homeo{\cube}$ and let $D\subset\cubein$ be countable dense. Assume that there is $z\in C_\filt$ such that for every $d\in D$ and $n\in\omega$,  $\lvert (d-h(d))(n)\rvert\leq\lvert z(n)\rvert$. Then $h[K_\filt]=K_\filt$.
 \end{lemma}
 \begin{proof}
  From the product topology and the density of $D$ it follows that for every $x\in\cube$ and $n\in\omega$, $\lvert (x-h(x))(n)\rvert\leq\lvert z(n)\rvert$. Also, given $m\in\omega$,
  $$
  \{n\in\omega:\lvert z(n)\rvert<2\sp{-m}\}\subset\{n\in\omega:\lvert(x-h(x))(n)\rvert<2\sp{-m}\}.
  $$
  So the result follows from Lemma \ref{homeo-restricts1}.
 \end{proof}
 
 Let $X\subset\cube$. We will say that $X$ is in \emph{general position} if given $x,y\in X$ with $x\neq y$ then $x(n)\neq y(n)$ for all $n\in\omega$. This definition was used by Stepr\={a}ns and Zhou in \cite{step-zhou} in their proof that $[0,1]\sp{\kappa}$ is \CDH{} under Martin's axiom. We need a version of \cite[Lemma 3.1]{step-zhou} that accounts for $C_\filt$, and this is what we prove in the next result.
 
 \begin{lemma}\label{lemma-gen_position}
  Let $\filt\subset\powerw$ be a free filter and let $D\subset\cubein$ be a countable set. Then there exists $h\in\homeo{\cube}$ such that $h[D]$ is in general position, $h[\cubein]=\cubein$ and $h[K_\filt]=K_\filt$.
 \end{lemma}
 \begin{proof}[Sketch of proof]
  We leave out details which can be easily checked by the reader. By the Inductive Convergence Criterion it is sufficient to recursively construct a sequence of homeomorphisms and define $h$ to be the limit of their composition.
  
  For each $n\in\omega$, let
  $$
  B_n=\{x\in\cube:x(n)\in\{-1,1\}\}.
  $$
  Then $B(\cube)=\bigcup\{B_n:n\in\omega\}$ and $B_n$ is closed for each $n<\omega$.
  \medskip
  
  \noindent{\it Claim}: If $F\subset\cubein$ is a finite set in general position and $p\in\cubein\setminus F$ then there is $f\in\homeo{\cube}$ such that $f[F\cup\{p\}]$ is in general position and $f[B_n]=B_n$ for each $n\in\omega$.\medskip
  
  In order to prove the Claim, in fact $f$ will be a limit of a sequence of homeomorphisms as we describe next. 
  
  Given $y\in F$ there exists $n_y<\omega$ such that $p(n_y)\neq y(n_y)$. Now fix any $n\in\omega$. Now consider the projection $\pi:\cube\to[-1,1]\sp{\{n_y,y\}}$. Then $\pi(y)$ and $\pi(p)$ are two different points of $[0,1]\sp{\{n_y,y\}}$ so there is a homeomorphism $g_{y,n}:[-1,1]\sp{\{n_y,y\}}\to[-1,1]\sp{\{n_y,y\}}$ such that $h_{y,n}(\pi(p))(n)\neq h_{y,n}(\pi(y))(n)$, and that is the identity in the boundary of $[-1,1]\sp{\{n_y,y\}}$. Then define 
  $$
  f_{y,n}(x)(i)=\left\{
  \begin{array}{ll}
  x(i), & \textrm{if } i\notin\{n_y,y\},\textrm{ and}\\
  g_{y,n}(\pi(x))(i), & \textrm{if } i\in\{n_y,y\}.
  \end{array}
  \right.$$
  so that $f_{y,n}\in\homeo{\cube}$. 
  
  Naturally, we may ask that each of these homeomorphisms is as close to $\id{\cube}$ as we may wish. Further, notice that each of these homeomorphisms leaves $B_n$ fixed set-wise for every $n\in\omega$.
  
  In order to define $f$ as required in the Claim, we need to take a sequence of homeomorphisms and take the limit of the compositions. However, as we take the limit, even though we are separating $F$ from $p$, in the limit it may be that we loose this separation in some coordinate. So we need to make sure that this does not happen. For example, if $x,y\in F\cup\{p\}$ and $n\in\omega$ are such that $x(n)\neq y(n)$, then in subsequent homeomorphisms we only have to make sure that the distance between $x(n)$ and $y(n)$ never lies below some positive threshold. 
  
  Thus, by taking an appropriate composition, we can obtain the homeomorphism $f$ as required in the Claim.
  
  By taking an appropriate sequence of homeomorphisms as in the Claim, it is not hard to find $h\in\homeo{\cube}$ such that $h[D]$ is in general position and $h[B_n]=B_n$ for each $n\in\omega$. This implies that $h[B(\cube)]=B(\cube)$ so $h[\cubein]=\cubein$.
  
  To prove that $h[K_\filt]=K_\filt$, we can refer to Lemma \ref{homeo-restricts1}. It is not hard to argue that it is sufficient to prove that for every $x\in\cube$ and $n\in\omega$, $\lvert h(x)(n)-x(n)\rvert<2\sp{-n}$. In order to achieve this, it is sufficient that we are careful when constructing the homeomorphisms given above. For example, $g_{y,n}$ can be chosen as close to the identity as required by the coordinates $y_n$ and $n$.
 \end{proof}
 
 \begin{lemma}\label{lemma-miss_a_ctble}
  Let $\filt\subset\powerw$ be a free filter and $D,E\subset\cubein$ be countable sets. Then there exists $h\in\homeo{\cube}$ such that $h[C_\filt]=C_\filt$, $h[\cubein]=\cubein$ and for all $d\in D$, $e\in E$ and $n\in\omega$, $h(d)(n)\neq e(n)$.
 \end{lemma}
 \begin{proof}
 Let $\pi_n:\cube\to[-1,1]$ be the projection into the $n$-th coordinate for each $n<\omega$.  Given $n<\omega$, there exists $h_n\in\homeo{[-1,1]}$ such that $h[\pi_n[D]]\cap \pi_n[E]=\emptyset$. Then if we define $h(x)(n)=h_n(x)$ for all $x\in\cube$ and $n\in\omega$, $h\in\homeo{\cube}$. As explained in the proof of Lemma \ref{lemma-gen_position}, we only have to choose $h_n$ for each $n\in\omega$ in such a way that for all $x\in[-1,1]$, $\lvert h_n(x)-x\rvert<2\sp{-n}$ so that we know $h[K_\filt]=K_\filt$. Notice that $h[\cubein]=\cubein$ is immediate by the definition of each $h_n$, $n\in\omega$.
 \end{proof}

 \noindent Consider the following statement, where $0<k<\omega$. Recall that $\partial{[-1,1]\sp{k}}$ denotes the topological boundary of $[-1,1]\sp{k}$.
 \begin{quote}
  $(\ast)_k$: Let $A,B\subset(-1,1)\sp{k}$ be finite sets with $A\cap B=\emptyset$, $\epsilon>0$, and $e:A\to B$ a bijection such that for every $a\in A$, \mbox{$\lVert a-e(a)\rVert<\epsilon$}. Then there exists $h\in\homeo{[-1,1]\sp{k}}$ such that for every $x\in[-1,1]\sp{k}$, $\lVert x-h(x)\lVert<\epsilon$, $h\restriction A=e$ and $h\restriction {\partial{[-1,1]\sp{k}}}=\id{\partial{[-1,1]\sp{k}}}$. 
 \end{quote}
 Statement $(\ast)_k$ is a generalization of strong $n$-homogeneity for all $1\leq n<\omega$ where we require that homeomorphisms are as small as possible. Clearly, $(\ast)_1$ is false because any homeomorphism of an interval either respects the order or inverts the order.
 
 \begin{lemma}
  The statement $(\ast)_k$ is true for every $k$ with $3\leq k<\omega$.
 \end{lemma}
 \begin{proof}
  Let $A=\{a_i:i<n\}$. Since a connected open set of $[-1,1]\sp{k}$ cannot be separated by sets of dimension $1$, it is possible to recursively construct a pairwise disjoint sequence of arcs $\{\alpha_i:i<n\}$ with diameter $<\epsilon$ such that $\alpha_i$ has endpoints $a_i$ and $e(a_i)$. Then it is possible to find a pairwise disjoint sequence of open sets $\{U_i:i<k\}$ such that $\alpha_i\subset U_i\subset\overline{U_i}\subset(-1,1)\sp{k}$ and $U_i$ has diameter $<\epsilon$ for each $i<k$. In fact, we may assume that $\overline{U_i}$ is homeomorphic to $[-1,1]\sp{k}$ for each $i<k$. Then by homogeneity of $k$-cells, it is possible to construct $h_i\in\homeo{\overline{U_i}}$ such that $h_i(a_i)=e(a_i)$. Define $h\in\homeo{[-1,1]}$ by $h(x)=h_i(x)$ if $x\in U_i$ and by $h(x)=x$ if $x\in[-1,1]\sp{k}\setminus\bigcup\{U_i:i< k\}$.
 \end{proof}

 \begin{ques}\label{small-case-open}
  Is $(\ast)_2$ true?
 \end{ques}
 
 The author personally asked Question \ref{small-case-open} to Professor Logan C. Hoehn during his 2019 visit to Mexico. Professor Hoehn conjectured that the answer is affirmative and provided an argument that, in the author's opinion, can be formalized as a proof. We will not try to reproduce Professor Hoehn's argument here as it is too extensive and not related to the main topic of this paper.
 
 Let $\mathcal{X}\subset[\omega]\sp\omega$. A tree $T\subset([\omega]\sp{<\omega})\sp{<\omega}$ is called a \emph{$\mathcal{X}$-tree of finite sets} if for each $s\in T$ there is $X_s\in\mathcal{X}$ such that for every $a\in[X_s]\sp{<\omega}$ we have $s\sp\frown a\in T$. It turns out that non-meager $P$-filters have a very useful combinatorial characterization as follows.

\begin{lemma}\cite[Lemma 1.3]{laflamme}\label{treelemma}
Let $\filt\subset\powerw$ be a free filter. Then $\filt$ is a non-meager $P$-filter if and only if every $\filt$-tree of finite sets has a branch whose union is in $\filt$.
\end{lemma}

So naturally, our strategy for the proof of sufficiency will be to recursively construct a homeomorphism, and (analogous to the proof in \cite{hg-hr-CDH}) to construct $z\in C_\filt$ with the properties given in the statement of Lemma \ref{homeo-restricts2}. However, we need some mechanism to connect $z$ with the Cantor set so we can use Lemma \ref{treelemma}.

Construct a Cantor scheme $\A=\{A_s:s\in 2\sp{<\omega}\}$ with the following properties
\begin{enumerate}[label=(\arabic*)]
 \item $A_\emptyset=[0,2]$,
 \item if $A_s=[a,b]$, then $A_{s\sp\frown 1}=[a,(a+b)/2]$ and $A_{s\sp\frown 0}=[(a+b)/2,b]$.
\end{enumerate}
Given $x\subset\omega\times\omega$ and $n\in\omega$, define
$$
\begin{array}{lrl}
x\sp{(n)}&=&\{i\in\omega: \pair{i,n}\in x\},\textrm{ and}\\
x_{(n)}&=&\{i\in\omega: \pair{n,i}\in x\}.
\end{array}
$$

We next define a function $\varphi:\mathcal{P}(\omega\times\omega)\to[0,2]\sp\omega$: given $x\subset\omega\times\omega$ and $m\in\omega$ we let $\phi(x)(m)$ be the unique point in $\bigcap\{A_{(x\sp{(m)}\restriction n)}:n\in\omega\}$. The proof of the following is left to the reader.

\begin{lemma}
Let $\varphi:\mathcal{P}(\omega\times\omega)\to[0,2]\sp\omega$ be defined as above.
 \begin{enumerate}[label=(\roman*)]
 \item For every $x\subset\omega\times\omega$ and $n\in\omega$, $0\leq\varphi(x)(n)\leq 2$.
  \item If $x\in \filt\sp{(\omega)}$, then $\frac{1}{2}\cdot\varphi(x)\in K_\filt$.
  \item If $x\subset\omega\times\omega$ is such that $\frac{1}{2}\cdot\varphi(x)\in K_\filt$ then $x\in\filt\sp{(\omega)}$.
  \item If $x\subset y\subset\omega\times\omega$ and $n\in\omega$, $\varphi(y)(n)\leq\varphi(x)(n)$.
 \end{enumerate}
\end{lemma}

We have all the ingredients necessary for our proof.

\begin{proof}[Proof of sufficiency in Theorem \ref{main-thm}]
Let $\filt\subset\powerw$ be a non-meager $P$-filter and let $D,E\subset C_\filt$ be countable dense sets, we must find a homeomorphism $h:\cube\to\cube$ such that $h[C_\filt]=C_\filt$ and $h[D]=E$. 

By Lemmas \ref{lemma-gen_position} and \ref{lemma-miss_a_ctble} we may assume that $D\cap E=\emptyset$ and $D\cup E$ is in general position. It is also possible to assume that the set $\{\lvert d(n)-e(n)\rvert:n\in\omega\}$ is disjoint from the countable set of endpoints of elements of the Cantor scheme $\A$. We consider an enumeration $D\cup E=\{d_n:n<\omega\}$.

Let $F\in\filt$ be such that $\omega\setminus F$ is infinite. We will assume that $F=\{3n:n\in\omega\}$ for the sake of a simpler proof. For each $n\in\omega$, we let
$$
J_n=[-1,1]\sp{\{3n,3n+1,3n+2\}}.
$$
For each $n\in\omega$, we define the projection $\pi_n:\cube\to J_n$ such that for all $x\in\cube$,  $\pi_n(x)=x\restriction\{3n,3n+1,3n+2\}$.

Recall that it is posssible to define a metric in $[-1,1]\sp{3}$ by
$$
\lVert x-y\rVert=\max\{\lvert x(0)-y(0)\rvert,\lvert x(1)-y(1)\rvert,\lvert x(2)-y(2)\rvert\}
$$
for $x,y\in[-1,1]\sp{3}$. We will be using this metric in each of the cubes $J_n$.

The reason we are grouping the coordinates in groups of three is that we will need property $(\ast)_3$. Notice that $K_\filt=K_{\filt\restriction F}\times[-1,1]\sp{\omega\setminus F}$, where $\filt\restriction F=\{x\cap F:x\in\filt\}$ is also a non-meager $P$-filter. So it is sufficient to have a condition like that of Lemma \ref{homeo-restricts2} only on the coordinates contained on $F$. In order to keep the proof simpler, we notice that it is sufficient to consider the following condition.

\begin{itemize}
 \item[$(\star)$] Let $h\in\homeo{\cube}$ such that $z\in\filt\sp{(\omega)}$ has the property that for all $d\in D$, $\lVert \pi_n(d)-\pi_n(h(d))\rVert\leq \varphi(z)(n)$. Then $h[K_\filt]=K_\filt$.
\end{itemize}

Let us give an outline of our proof. We will construct a $\filt\sp{(\omega)}$-tree of finite sets and we will have a homeomorphism that is a product of homeomorphisms in the spaces $\{J_n:n\in\omega\}$ at each node. Each branch of the tree will define a homeomorphism $h\in\homeo{Q}$ such that $h[\cubein]=\cubein$ and $h[D]=E$. By Lemma \ref{treelemma}, there will exist one branch such that $h$ also has the property that $h[K_\filt]=K_\filt$.

Let us be more explicit now. We recursively define $T\subset([\omega\times\omega]\sp{<\omega})\sp{<\omega}$ and for each $s\in T$ we define $m_s\in\omega$, $F_s\in\filt\sp{(\omega)}$, $f_s\subset D\times E$ and a collection of homeomorphisms $\G_s=\{h_i\sp{(s)}:i\in\omega\}$ with the following properties.
\begin{enumerate}[label=(\roman*)]
 \item $\emptyset\in T$, $m_\emptyset=0$, $F_\emptyset=\omega\times\omega$ and $\G_\emptyset=\{\id{J_i}:i\in\omega\}$.
 \item The immediate successors of $s\in T$ are $\{s\cup\{\pair{\lvert s\rvert,t}\}:t\in[F_s]\sp{<\omega}\}$.
 \item For all $s\in T$, $F_s\subset \omega\times(\omega\setminus m_s)$.
 \item If $s\in T$ and $i<\lvert s\rvert$, $s(i)\subset\omega\times m_s$.
 \item If $s\in T$ with $n=\lvert s\rvert$, then $f_s$ is a finite bijection with $\mathrm{dom}(f_s)\subset D$, $\mathrm{im}(f_s)\subset E$ and $d_n\in\mathrm{dom}(f_s)\cup\mathrm{im}(f_s)$.
 \item If $s,t\in T$ are such that $s\subset t$, then $f_s\subset f_t$.
 \item If $s\in T$ and $i\in\omega$ then $h_i\in\homeo{J_i}$ with $h_i\restriction{\partial{J_i}}=\id{\partial{J_i}}$.
 \item If $s\in T$ and $m_s\leq i$ then $h_i=\id{J_i}$.
 \item Let $s\in T$, $\lvert s\rvert=n$, $i<m_s$ and $k\in\omega$ with $d_k\in\mathrm{dom}(f_s)\cup\mathrm{im}(f_s)$. If $d_k\in D$, then $$\pi_i(f_s(d_k))=(h_i\sp{s}\circ h_i\sp{s\restriction{n-1}}\circ\cdots\circ h_i\sp{s\restriction 1}\circ h_i\sp{\emptyset})(\pi_i(d_k)). $$ If $d_k\in E$ then, $$\pi_i(f_s\sp{-1}(d_k))=(h_i\sp{s}\circ h_i\sp{s\restriction{n-1}}\circ\cdots\circ h_i\sp{s\restriction 1}\circ h_i\sp{\emptyset})\sp{-1}(\pi_i(d_k)). $$
 \item Let $s\in T$ and $i<m_s$. For every $x\in J_i$ we have the inequality
 $$\lVert x-(h_i\sp{s}\circ h_i\sp{s\restriction{n-1}}\circ\cdots\circ h_i\sp{s\restriction 1}\circ h_i\sp{\emptyset})(x)\rVert< \varphi\left(\bigcup\{s(j):j<m_s\}\right)(i).$$
 \item Let $s\in T$, $k<\lvert s\rvert$ and $m_s\leq i$. If $d_k\in D$, then $$\lVert\pi_i(d_k)- \pi_i(f_s(d_k))\rVert\leq\varphi(F_s)(i).$$ If $d_k\in E$, then $$\lVert\pi_i(d_k)- \pi_i(f_s\sp{-1}(d_k))\rVert\leq\varphi(F_s)(i).$$
\end{enumerate}

We also assume that for each $s\in T$ and $i<m_s$, the sequence of homeomorphisms $\{h_i\sp{s\restriction m}:m\leq \lvert s\rvert\}$ is chosen to satisfy the Inductive Convergence Criterion so that every branch will converge to a homeomorphism.

Before proving the construction is possible, let us notice an important consequence of items (iii) and (iv). Recursively, they imply that if $b$ is a branch of $T$, then $\{b(n):n<\omega\}$ is pairwise disjoint. But no only this, but $b(n+1)\subset\omega\times(m_{b\restriction n+1}\setminus m_{b\restriction n})$ for all $n\in\omega$.

So now let us assume that we have constructed the tree up to some $s\in T$ with $\mathrm{dom}(s)=n$ and, as stated in (ii), given an immediate successor $s\sp\prime=s\cup\{\pair{\lvert s\rvert,t}\}$ where $t$ is a finite subset of $F_s$. First, we define $$m_{s\sp\prime}=\max\{j\in\omega:\exists i\in\omega\ (\pair{i,j}\in t)\}+m_s+1.$$

Next, we define $f_{s\sp\prime}$. Let $\ell=\min\{i\in\omega: d_i\notin\mathrm{dom}(f_s)\cup\mathrm{im}(f_s)\}$. We may assume that $d_\ell\in D$; the case when $d_\ell\in E$ will be similar. So we have to choose the value of $f_{s\sp\prime}(d_\ell)\in E$. We have to choose this element so that we do not break the promises we have made in conditions (ix), (x) and (xi). Thus, for each $i<m_{s\sp\prime}$ we will define an open set $V_i$ of $J_i$ where we will want $f_{s\sp\prime}(d_\ell)\in E$ to be.\medskip

\noindent\underline{Case 1:} $i<m_s$. Let $g=(h_i\sp{s}\circ h_i\sp{s\restriction{n-1}}\circ\cdots\circ h_i\sp{s\restriction 1}\circ h_i\sp{\emptyset})$, $\delta_i=\varphi\left(\bigcup\{s(j):j<m_s\}\right)(i)$ and $M=\lVert \pi_i(d_\ell)-g(\pi_i(d_\ell))\lVert$. By condition (x) and continuity, there exists an open set $U_i\subset J_i$ such that $\pi_i(d_\ell)\in U$ and for all $x\in U_i$, 
$$
\lVert x-g(x)\rVert\in(M_i/2,(M_i+\delta_i)/2).
$$
Now, let $V_i$ be an open subset of $J_i$ with $g(\pi_i(d_\ell))\in V_i$, $V_i\cap\textrm{im}(f_i)=\emptyset$ and $V_i\subset g[U_i]$. We may assume that $\overline{V_i}$ is homeomorphic to $[-1,1]\sp{3}$.\medskip

\noindent\underline{Case 2:} $m_s\leq i<m_{s\sp\prime}$. By the definition of $s\sp\prime$, $t$ and the function $\varphi$ it is easy to see that $\varphi(F_s)(i)\leq\varphi(t)(i)$. Here, it is important to notice that since $t$ is a finite set, $\varphi(t)(i)$ is an endpoint of some (and many) of the elements of the Cantor scheme $\A$. Further, by the choices in $D$ and $E$ that we made in the first step in the proof, it follows that the distances on the left sides of the equations given in (xi) cannot be equal to $\varphi(t)(i)$. Then it follows that for $k<n$ we have the following strict inequalities:  if $d_k\in D$, then $$\lVert\pi_i(d_k)- \pi_i(f_s(d_k))\rVert<\varphi(t)(i),$$ and if $d_k\in E$, then $$\lVert\pi_i(d_k)- \pi_i(f_s\sp{-1}(d_k))\rVert<\varphi(t)(i).$$ We define $V_k=\{x\in J_i:\lVert x-\pi_i(d_k)\rVert<\varphi(t)(i)\}$. \medskip

Then we define $V=\{x\in\cubein:\forall i<m_{s\sp\prime}\ (\pi_i(x)\in V_i)\}$, choose $e\in V\cap E$ and we define $f_{s\sp\prime}=f_s\cup\{\pair{d_\ell,e}\}$. It recursively follows that (v) is satisfied.

Next, we have to define $\G_{s\sp\prime}$. For $i\geq m_{s\sp\prime}$, condition (viii) defines $h_s\sp{s\sp\prime}$ so let us focus on $i< m_{s\sp\prime}$.

First, let $i< m_s$ and consider the definitions given above in Case 1. Then define $h_i\sp{s\sp\prime}\in\homeo{J_i}$ to be such that $h_i\sp{s\sp\prime}(g(\pi_i(d_\ell)))=\pi_i(e)$ and $h_i\sp{s \sp\prime}\restriction{J_i\setminus V_i}=\id{J_i\setminus V_i}$. It is easy to see that (ix) and (x) will hold for $s\sp\prime$. Notice also that we can take $V_i$ as small as we wish, and this will give us $h_i\sp{s\sp\prime}$ close to $\id{J_i}$ as needed by the Inductive Convergence Criterion.

Now take $m_s\leq i<m_{s\sp\prime}$. Define $B_0=\{k<\omega:d_k\in\mathrm{dom}(f_{s})\}$ and $B_1=\{k<\omega:d_k\in\mathrm{im}(f_{s}\}$; then consider the two finite sets  
$$
\begin{array}{ccl}
 H_0 & = &\{\pi_i(d_k):k\in B_0\}\cup\{\pi_i(f_s(d_k)):k\in B_1\}\cup\{\pi_i(d_\ell)\}\\
 H_1 & = &\{\pi_i(d_k):k\in B_1\}\cup\{\pi_i(f_s\sp{-1}(d_k)):k\in B_0\}\cup\{\pi_i(e)\}
\end{array}
$$
Then $f_{s\sp\prime}$ naturally induces a bijection $r:H_0\to H_1$. Notice that here we are using our assumptions about $D$ and $E$ at the begining of the proof. Now, according to what was said in Case 2, for every $a\in H_0$ we have the inequality $\lVert a-r(a)\rVert<\varphi(t)(i)$. Here we use property $(\ast)_3$ that we defined above. Let $h_i\sp{s\sp\prime}\in\homeo{J_i}$ extend $r$ in such a way that for all $x\in J_i$, $\lVert x- h_i\sp{s\sp\prime}(x)\rVert<\varphi(t)(i)$ and $h_i\sp{s\sp\prime}\restriction{\partial J_i}=\id{\partial J_i}$. Notice that, by the observations made before $\varphi(t)=\varphi(\bigcup\{s\sp\prime(j):j<m_{s\sp\prime}\})(i)$. This implies that items (ix) and (x) hold for $s\sp\prime$.

Finally, we have to define $F_{s\sp\prime}$. Consider $d_k-e$, since both $d_k,e\in C_\filt$ it is easy to see that $\frac{1}{2}\cdot (d_k-e)\in K_\filt$. By the choice of $D$ and $E$ at the begining of the proof, it is easy to see that $\lvert\varphi\sp\leftarrow(d_k-e)\rvert=1$. Let $Y\in\mathcal{P}(\omega\times\omega)$ be such that $\varphi(Y)=d_k-e$. Then $Y\in\filt\sp{(\omega)}$. Define
$$
F_{s\sp\prime}=F_s\cap Y\cap\left(\omega\times(\omega\setminus m_{s\sp\prime})\right).
$$
The proof that the rest of the conditions hold with this definition is not hard. So this completes the recursive construction of the tree.

Let $b$ be a branch of $T$. For each $i\in\omega$ we define
$$
h_i\sp{b}=\lim_{n\to\infty}{(h_i\sp{b\restriction n}\circ h_i\sp{b\restriction{n-1}}\circ\cdots\circ h_i\sp{b\restriction 1}\circ h_i\sp{\emptyset})}
$$
which is a homeomorphism as we have discussed earlier. So let $h\sp{b}=\prod_{i<\omega}{h_i\sp{b}}\in\homeo{\cube}$ be the product homeomorphism. It easily follows that $h[D]=E$. Also, $h[B(\cube)]=B(\cube)$ holds because each $h_i\sp{b}\restriction{\partial J_i}=\id{\partial J_i}$ for each $i\in\omega$.

So finally, since $\filt\sp{(\omega)}$ is a non-meager $P$-filter, there exists a branch $y$ of $T$ whose union $z=\bigcup\{y(i):i\in\omega\}$ belongs to $\filt\sp{(\omega)}$. It is now easy to verify that the hypothesis in property $(\star)$ holds. Thus, $h\sp{y}[K_\filt]=K_\filt$. Then $h\sp{y}[C_\filt]=C_\filt$, which completes the proof.
\end{proof}

\section{A proof that $[0,1]\sp{\kappa}$ is \CDH{} when $\kappa<\mathfrak{p}$}\label{correction}

As mentioned in the Introduction, in this section we provide a correct proof of the Stepr\={a}ns-Zhou result that $X\sp\kappa$ is \CDH{} when $X\in\{[0,1], 2\sp\omega,\R\}$ and $\kappa<\mathfrak{p}$.

Naturally, this proof is simpler because we do not need to consider any filter when constructing a homeomorphism. Moreover, we use an idea similar to the one in the previous section. Namely, $[0,1]$ is not rich enough and we need to consider homeomorphisms of $[0,1]\sp{k}$ for some $k$ with $1<k<\omega$. In the previous section we needed property $(\ast)_k$ that we know holds if $k$ is at least 3. For the proof in this section, we only need to work in $[0,1]\sp2$. 

This is because $\R\sp2$ is \emph{strongly $2$-homogeneous}; as $2\sp\omega$ also is. And in $[0,1]\sp{2}$ we have a similar property: given $x,y,z,w\in(0,1)\sp2$ with $x\neq y$ and $z\neq w$ there exists a homeomorphism $h:[0,1]\sp2\to[0,1]\sp2$ with $h(x)=z$, $h(y)=w$ and $h\restriction{\partial{[0,1]\sp2}}=\id{\partial[0,1]\sp2}$.

We will leave most of the details of this proof to the reader. Also, we will work with the case $X=[0,1]$; the other cases are similar and will be left to the reader. For each $\alpha<\kappa$, let $J_\alpha=[0,1]\sp{2}$, we will prove that $\Pi=\prod_{\alpha<\kappa}{J_\alpha}$ is \CDH{}.

Given $f,g\in\homeo{[0,1]\sp2}$, let $\rho(f,g)=\sup\{\lVert f(x)-g(x)\rVert\colon x\in [0,1]\sp2\}$ and $\varsigma(f,g)=\max\{\rho(f,g),\rho(f\sp{-1},g\sp{-1})\}$. Then $\varsigma$ is a compatible complete metric for $\homeo{[0,1]\sp2}$. This discussion can be found in \cite{vminf89}.

So let $D,E\subset\Pi$ be countable dense. As mentioned in \cite{step-zhou} we may assume that $D\cup E$ is in general position, $D\cap E=\emptyset$ and for every $x\in D\cup E$ and $\alpha<\kappa$ we have $x(\alpha)\in (0,1)\sp2$. We will define a poset to force a homeomorphism $H:\Pi\to\Pi$ with $H[D]=E$.

We define a poset $\poset$ inspired by the poset in Baldwin's paper \cite{baldwin-MA_kappa_homogeneous}. A condition $p\in\poset$ is $p=\pair{f_p,n_p,F_p,\{h\sp{p}_\alpha:\alpha<\kappa\}}$ where
\begin{enumerate}[label=(\arabic*)]
 \item $f_p$ is a finite injection with $\mathrm{dom}(f_p)\subset D$ and $\mathrm{im}(f_p)\subset E$,
 \item $n_p\in\omega$,
 \item $F_p\in[\kappa]\sp{<\omega}$,
 \item for all $\alpha<\kappa$, $h\sp{p}_\alpha\in\homeo{J_\alpha}$ with $h\sp{p}_\alpha\restriction{\partial J_\alpha}=\id{\partial J_\alpha}$,
 \item for all $\alpha\in F_p$ and $d\in\mathrm{dom}(f_p)$, $h\sp{p}_\alpha(d(\alpha))=f_p(d)(\alpha)$, and
\item if $\alpha\in\kappa\setminus F_p$ then $h_\alpha\sp{p}=\id{J_\alpha}$.
\end{enumerate}
We define $q\leq p$ if
\begin{enumerate}[resume]
 \item $f_p\subset f_q$,
 \item $n_p\leq n_q$,
 \item $F_p\subset F_q$, and
 \item for all $\alpha\in F_p$ and $d\in\mathrm{dom}(F_p)$, $\varsigma(h_\alpha\sp{p},h_\alpha\sp{q})<2\sp{-n_p}-2\sp{-n_q}$.
\end{enumerate}

First, we give the proof that $\poset$ is $\sigma$-centered. The following argument is inspired by Baldwin's \cite{baldwin-MA_kappa_homogeneous}. First, if $Y\subset[0,1]\sp2$, then 
$$
\delta(Y)=\min\{\lVert x-y\rVert:x,y\in Y, x\neq y\}.
$$
where $\lVert x\rVert$ is the usual Euclidean distance (or any other equivalent one). Let $S$ be the set of finite injections $s$ with $\mathrm{dom}(s)\subset D$, $\mathrm{im}(s)\subset E$.

Now, all of our homeomorphisms will be products of homeomorphisms so consider the product space $\Upsilon=\prod\{\homeo{J_\alpha}:\alpha<\kappa\}$, then this space is separable by the Hewitt-Marczewski-Pondiczery Theorem. Thus, there is a countable dense set $H\subset\Upsilon$.

For $s\in S$, $m,k\in\omega$ and $\tau\in H$ we define
$$
\begin{array}{ccl}
\poset(s,m,k,\tau) & = &\left\{p\in\poset: f_p=s,\ n_p=m,\right.\\
& & \ \forall \alpha\in F_p\ (2\sp{-(k+m)}<\delta(\{e(\alpha):e\in\mathrm{im}(F_p)\})),\\ 
& & \ \left. \forall \alpha\in F_p\ (\varsigma(h_\alpha\sp{p},\tau_\alpha)<2\sp{-(k+m+3)})\right\}.
\end{array}
$$
It is easy to see that
$$
\poset=\bigcup\{\poset(s,m,k,\tau):s\in S, m,k\in\omega, \tau\in H\}.
$$

Now fix $s\in S$, $m,k\in\omega$ and $\tau\in H$; we will prove that $\poset(s,m,k,\tau)$ is centered. Take a finite subset $\{p_i:i<n\}\subset\poset(s,m,k,\tau)$. Define $F\sp\prime=\bigcup\{F_{p_i}:i\in\omega\}$, which is a finite subset of $\kappa$.

Fix $\alpha\in F\sp\prime$. For each $d\in\mathrm{dom}(s)$, we define $$B(d,\alpha)=\{x\in J_\alpha: \lVert \tau_\alpha(d(\alpha))-x\lVert\leq2\sp{-(k+m+3)}\}.$$ Notice that $\tau_\alpha(d(\alpha))$ and $s(d)(\alpha)$ are both elements of $B(d,\alpha)$. Also, it is not hard to see that $\{B(d,\alpha):d\in D\}$ is pairwise disjoint. Define $g_\alpha\in\homeo{[0,1]\sp2}$ in such a way that $g_\alpha(\tau_\alpha(d(\alpha)))=s(d)(\alpha)$ for each $d\in\mathrm{dom}(s)$ and $g_\alpha$ restricted to $J_\alpha\setminus\bigcup\{B(d,\alpha):d\in\mathrm{dom}(s)\}$ is the identity. Clearly, we may also assume that $g_\alpha\restriction{\partial [0,1]\sp2}=\id{\partial [0,1]\sp2}$.

Let us define $q\in\poset$ that will be a common extension for all $\{p_i:i<n\}$. Let $f_q=s$, $n_q=n_p+1$ and $F_q=F\sp\prime$. If $\alpha\in F_q$, let $h_\alpha\sp{q}=g_\alpha\circ\tau_\alpha$. Finally, if $\alpha\in\kappa\setminus F_q$ then $h_\alpha\sp{q}=\id{J_\alpha}$. That this is indeed a common extension will be left to the reader. This completes the proof that $\poset$ is $\sigma$-centered.

So now we have to consider the following dense subsets of $\poset$.
$$
\begin{array}{ccll}
 P(d) & = & \{p\in\poset: d\in\mathrm{dom}(f_p)\} & \textrm{for }d\in D, \\
 Q(e) & = & \{p\in\poset: e\in\mathrm{im}(f_p)\} & \textrm{for }e\in E, \\
 R(n) & = & \{p\in\poset: n_p\geq n\} & \textrm{for }n\in\omega, \textrm{ and}\\
 S(\alpha) & = & \{p\in\poset: \alpha\in F_p\} & \textrm{for }\alpha<\kappa.
\end{array}
$$

We will leave the proof of density to the reader, who will surely find this proofs similar to the arguments in the previous section. In particular, in order to prove the density of the first two families of sets, the homogeneity property of $[0,1]\sp{2}$ mentioned above must be used.

Assume that $G$ is a filter in $\poset$ that intersects all dense sets above. First, it easily follows that $f_G=\bigcup\{f_p:p\in G\}$ is a bijection from $D$ to $E$. It only remains to prove that $f_G$ extends to a homeomorphism of $\Pi$. In order to achieve this, we need to prove that the homeomorphisms in each condition converge to some homeomorphism.

Fix $\alpha<\kappa$ and let $p_0\in G$ be such that $\alpha\in F_{p_0}$. For each $n>n_{p_0}$ we define
$$
K_n=\{h_\alpha\sp{p}:p\in G, p\leq p_0,n<n_p\},
$$
which is a subset of $\homeo{J_\alpha}$. Let $p,q\in G$ with $p,q\leq p_0$. Then there exists $r\in G$ with $r\leq p$ and $r\leq q$. By the definition of the partial order, $\varsigma(h_\alpha\sp{p}, h_\alpha\sp{q})\leq \varsigma(h_\alpha\sp{p},h_\alpha\sp{r})+\varsigma(h_\alpha\sp{r},h_\alpha\sp{q})<2\sp{-n_p}+2\sp{-n_q}\leq 2\sp{-n}$. This proves that $K_n$ has diameter $\leq 2\sp{-n}$ with the metric $\varsigma$. Define $h_\alpha\sp{G}$ to be the unique point in the intersection
$$
\bigcap\{\overline{K_n}:n>n_{p_0}\}.
$$
Then $h_\alpha\sp{G}\in\homeo{J_\alpha}$.

Finally, define $h\sp{G}=\prod\{h_\alpha\sp{G}:\alpha<\kappa\}$. It follows that $h_G$ is a homeomorphism. We leave the proof that $f_G\subset h_G$ to the reader. This concludes the proof.

\section{Higher cardinalities}

In this section we pose questions about higher cardinalities. From the Hewitt-Marczewski-Pondiczery Theorem, we know that $2\sp\kappa$ is separable if and only if $\kappa\leq\cont$. So we may ask for an extension of \cite[Theorem 4]{step-zhou}.

\begin{ques}(Michael Hru\v s\'ak)
 Find all free filters $\filt\subset\powerw$ and cardinals $\omega\leq\kappa\leq\cont$ such that $\filt\sp{\kappa}$ is \CDH{}.
\end{ques}

Naturally, the next question is whether there is an uncountable space $X$ such that $C_p(X)$ is \CDH{}. Here we nominate a specific space for this role.

Let $\omega\leq\kappa\leq\cont$ and $\filt$ be a free filter on $\omega$. We define a space $X_{\filt,\kappa}=(\kappa\times\omega)\cup\{\infty\}$ such that $\kappa\times \omega$ is discrete and a neighborhood of $\infty$ is of the form $\{\infty\}\cup\left(\bigcup\{\alpha\times A_\alpha:\alpha<\kappa\}\right)$ such that $A_\alpha\in\filt$ for all $\alpha<\kappa$.

\begin{lemma}
 If $\filt$ is a free filter on $\omega$ and $\omega\leq\kappa\leq\cont$, then $C_p(X_{\filt,\kappa})$ is separable.
\end{lemma}
\begin{proof}
 Choose $p\in\R$. For each $\alpha<\kappa$, let $S_\alpha\subset\R\setminus\{p\}$ be a sequence converging to $p$. It is easy to choose these sequences in such a way that $S_\alpha\cap S_\beta=\emptyset$ if $\alpha\neq\beta$. Let $Y=\{p\}\cup\left(\bigcup\{S_\alpha:\alpha<\kappa\}\right)$. Then it easily follows that $Y$ is a continuous, one-to-one image of $X_{\filt,\kappa}$. From Theorem \ref{Cp-facts}, $C_p(X_{\filt,\kappa})$ is separable.
\end{proof}

We leave the proof of the following fact to the reader.

\begin{lemma}
 If $\filt$ is a free filter on $\omega$ and $\omega\leq\kappa\leq\cont$, then $C_p(X_{\filt,\kappa})$ is homeomorphic to $(C_\filt)\sp\kappa$
\end{lemma}

So the natural question for function spaces is the following.

\begin{ques}
 Find all free filters $\filt\subset\powerw$ and cardinals $\omega\leq\kappa\leq\cont$ such that $C_p(X_{\filt,\kappa})$ is \CDH{}.
\end{ques}

\section{Osipov's remarks}

After submitting this paper for publication, Professor Alexander V. Osipov contacted the author regarding his paper \cite{osipov}, which has some results that provide further insight to Question \ref{question-metrizable}. We refer the reader to Professor Osipov's paper \cite{osipov} from which the following can be obtained.
 
 \begin{coro}
  Let $X$ be separable and metrizable such that $C_p(X)$ is \CDH{}. Then $X$ is a $\gamma$-set.
 \end{coro}
 \begin{proof}
  If $X$ is separable and metrizable then $C_p(X)$ is sequentially separable (see \cite{osipov}). If $C_p(X)$ is also \CDH{}, then it is strongly sequentially separable (\cite{osipov}). By the theorems in \cite{osipov} it follows that $X$ is a $\gamma$-set.
 \end{proof}
 
 \begin{ques}
  Is there an uncountable $\gamma$-set X such that $C_p(X)$ is \CDH{}?
 \end{ques}
 
 Also, notice that the main result of this paper answers \cite[Problem 2]{osipov} in the negative.


\begin{thebibliography}{99}

\bibitem{arkh-top_func_spaces}  Arkhangel'ski\u{i}, A. V.; ``Topological Function Spaces.'' Translated from the Russian by R. A. M. Hoksbergen. Mathematics and its Applications (Soviet Series), 78. Kluwer Academic Publishers Group, Dordrecht, 1992. x+205 pp. ISBN: 0-7923-1531-6

\bibitem{arh-vm-homogeneity} Arhangel'ski\u \i, A. V. and van Mill, J.; ``Topological Homogeneity.'', in: Recent progress in general topology III, K. P. ~Hart et al. (eds.), Atlantis Press, Paris, 2014, 1--68.

\bibitem{baldwin-MA_kappa_homogeneous} Baldwin, S.; ``Martin's axiom and $\kappa$-dense homogeneity.'' preprint

\bibitem{bart} Bartoszy\'nski, T.; Judah, H.; ``Set Theory. On the Structure of the Real Line.'' A K Peters, Ltd., Wellesley, MA, 1995. xii+546 pp. ISBN: 1-56881-044-X

\bibitem{FarahHrusakMartinez05} Farah, I.,; Hru{\v{s}}{\'a}k, M.; Mart\'{\i}nez Ranero, C.; ``A countable  dense homogeneous set of reals of size $\aleph_1$.'', Fund. Math. 186 (2005), 71--77.

\bibitem{openprobCDH} Fitzpatrick Jr., B.; Zhou, H-X.; ``Some Open Problems in Densely Homogeneous Spaces.'', in Open Problems in Topology (ed. J. van Mill and M. Reed), 1984, pp. 251--259, North-Holland, Amsterdam.

\bibitem{fitz_zhou-CDH_baire} Fitzpatrick, B., Jr.; Zhou, H. X.; ``Countable dense homogeneity and the Baire property.'' Topology Appl. 43 (1992), no. 1, 1--14.
 
 \bibitem{hg-hr-CDH} Hern\'andez-Guti\'errez, R.; Hru\v s\'ak, M; ``Non-meager $P$-filters are countable dense homogeneous.'' Colloq. Math. 130 (2013), 281--289.

\bibitem{hg-hr-vm} Hern\'andez-Guti\'errez, R.; Hru\v s\'ak, M.; van Mill, J.; ``Countable dense homogeneity and $\lambda$-sets.'', Fund. Math. 226 (2014), no. 2, 157--172.

\bibitem{hru-vm-cdh_survey} Hru\v s\'ak, M.; van Mill, J.; ``Open Problems on Countable Dense Homogeneity.'', Topology Appl. {241} (2018), 185--196.

\bibitem{hrus-zam} Hru\v s\'ak, M.; Zamora Avil\'es, B.; ``Countable dense homogeneity of definable spaces.'' Proc. Amer. Math. Soc. 133 (2005), no. 11, 3429--3435.

\bibitem{kunen-medini-zdomskyy} Kunen, K.; Medini, A.; Zdomskyy, L.; ``Seven characterizations of non-meager $P$-filters.'' Fund. Math. 231 (2015), no. 2, 189--208.

\bibitem{marciszewski-analytic_coanalytic} Marciszewski, W.; ``On analytic and coanalytic function spaces $C_p(X)$.'' Topology Appl. 50 (1993), 341--248.

\bibitem{laflamme} Laflamme, C.; ``Filter games and combinatorial properties of strategies.'' Set Theory (Boise, ID, 1992-1994), 51--67,
Contemp. Math., 192, Amer. Math. Soc., Providence, RI, 1996.


\bibitem{marciszewski-Pfilters} Marciszewski, W.; ``P-filters and hereditary Baire function spaces.'' Topology Appl. 89 (1998), no. 3, 241--247.

\bibitem{vminf89} van Mill, J.; ``Infinite-Dimensional Topology. Prerequisites and Introduction.'' North-Holland Mathematical Library, 43. North-Holland Publishing Co., Amsterdam, 1989. xii+401 pp. ISBN: 0-444-87133-0

 \bibitem{vm-inf_dim_funct_spaces} van Mill, J.; ``The Infinite-Dimensional Topology of Function Spaces.'' North-Holland Mathematical Library, 64. North-Holland Publishing Co., Amsterdam, 2001. xii+630 pp. ISBN: 0-444-50557-1
 
 \bibitem{osipov} Osipov, A.V.; ``Application of selection principles in the study of the properties of function spaces.'' Acta Math. Hungar. 154 (2018), no. 2, 362--377. 
 
  \bibitem{shelah-proper} Shelah, S.; ``Proper and Improper Forcing.'' Second edition. Perspectives in Mathematical Logic. Springer-Verlag, Berlin, 1998. xlviii+1020 pp. ISBN: 3-540-51700-6 

\bibitem{step-zhou} Stepr\={a}ns, J.; Zhou, H. X.; ``Some results on CDH spaces. I.'' Special issue on set-theoretic topology. Topology Appl. 28 (1988), no. 2, 147--154.

 \bibitem{tk-no_conv_seq} Tkachuk, V. V.; ``Many Eberlein-Grothendieck spaces have no non-trivial convergent sequences.'' Eur. J. Math. 4 (2018), no. 2, 664--675.

\end{thebibliography}
\end{document}